\documentclass[12pt,reqno]{amsart}
\usepackage{geometry}
\usepackage[latin1]{inputenc}
\usepackage[italian,english]{babel}
\usepackage{amsmath, amsfonts, amsthm, xcolor}
\usepackage{paralist}
\numberwithin{equation}{section}
\usepackage{mathtools, mathabx, verbatim}

\newtheorem{thm}{Theorem}[section]

\newtheorem{prop}[thm]{Proposition}

\newtheorem{defn}[thm]{Definition}
\theoremstyle{definition}
\newtheorem{rem}[thm]{Remark}
\theoremstyle{remark}

\newcommand{\ds}{\displaystyle}

\newcommand{\R}{\mathbb{R}}

\DeclareMathOperator{\esssup}{ess\,sup}

{\left\{\begin{array}{@{}l@{}}}{\end{array}\right.}
\patchcmd{\abstract}{\scshape\abstractname}{\textbf{\abstractname}}{}{}
\makeatletter %note a di pagina senza numero 1
\def\@makefnmark{} %note a di pagina senza numero 2
\makeatother %note a di pagina senza numero 3

\title{On a weighted anisotropic eigenvalue problem}
\author[N. Gavitone, R. Sannipoli]{
	Nunzia Gavitone, Rossano Sannipoli}
\address{Dipartimento di Matematica e Applicazioni ``R. Caccioppoli'', Universit\`a degli studi di Napoli Federico II \\ Via Cintia, Complesso Universitario Monte S. Angelo, 80126 Napoli, Italy.}
\email{nunzia.gavitone@unina.it}
\address{Dipartimento di Matematica, Universit\`a di Pisa \\ Largo B. Pontecorvo 5, 56127 Pisa, Italy.}
\email{rossano.sannipoli@dm.unipi.it}

\begin{document}
\maketitle
\begin{abstract}
In this paper we deal with a  weighted eigenvalue problem for the anisotropic $(p,q)$-Laplacian with Dirichlet boundary conditions. We study the main properties of the first eigenvalue and  a reverse H\"older type inequality for the corresponding eigenfunctions. 
\end{abstract}
\section{Introduction}
Let $\Omega\subset \mathbb{R}^n$, $n \ge 2$, be an open, bounded and connected set and let  $p, q$ be such that $1<p$, and  $1<q<p^*$, where $p^*=np/(n-p)$ if $p<n$ and $p^*=\infty$ if $p\ge n$. In this paper we study the following  variational problem 
\begin{equation} \label{varcar}
    \lambda^H_{p,q}(\Omega) = \inf_{\substack{u\in W^{1,p}_0(\Omega), \\ u\neq 0}}\frac{\ds\int_\Omega H(\nabla u)^p\,dx}{\bigg(\ds\int_\Omega m|u|^q\,dx\bigg)^{\frac{p}{q}}},
\end{equation}
where $H$ is a sufficiently smooth norm in $\mathbb{R}^n$ and $m\in L^{\infty}(\Omega)$ is a positive function. Obviously $\lambda^H_{p,q}(\Omega)$ depends also by $m$,  but to simplify the notation we will omit its dependence.  

 The Euler-Lagrange equation associated with the minimization problem \eqref{varcar} is the following 
weighted eigenvalue problem for the anisotropic $(p,q)-$Laplace operator with Dirichlet boundary condition
\begin{equation} \label{APQ}
\begin{cases}
-\mathcal{L}_{p}(u)=\lambda m(x)\| u\|_{q,m}^{p-q}|u|^{q-2}u & \mbox{in}\ \Omega\vspace{0.2cm}\\
u=0&\mbox{on}\ \partial \Omega\vspace{0.2cm},
\end{cases}
\end{equation}
where  $\| u\|_{q,m}=\| u\|_{L^q(\Omega,m)}$ is the weighted Lebesgue norm of $u$ and $\mathcal{L}_p$ is the so-called anisotropic $p$-Laplacian operator  defined as follows
\begin{equation}
    \mathcal{L}_{p}(u) = \mathrm{div} (H(\nabla u)^{p-1}H_{\xi}(\nabla u)).
\end{equation}

We stress that when $p=q$ and $m(x)\equiv 1$, \eqref{varcar} is the first eigenvalue $\lambda^H_p(\Omega)$ of the anisotropic $p$-Laplacian and it has been studied by many authors (see for instance \cite{bfk},\cite{mana} and the references therein). In particular in \cite{bfk} it is proved that $\lambda^{H}_p(\Omega)$ is simple for any $p$, the corresponding eigenfunctions have a sign and that a suitable Faber-Krahn inequality holds.\\
%Actually, also in the general case $p \neq q$, $\lambda^H_{p,q}(\Omega)$ verifies analogous properties.\\ 
When $H=\mathcal E$ is the usual Euclidean norm, $ \mathcal{L}_{p}(u)$ is the well-known $p$-Laplace operator and the eigenvalue problem \eqref{APQ} reduces to the following
\begin{equation} \label{pqe}
\begin{cases}
-\Delta_p u=\lambda \|u\|_{q,m}^{p-q}|u|^{q-2}u & \mbox{in}\ \Omega\vspace{0.2cm}\\
u=0&\mbox{on}\ \partial \Omega.\vspace{0.2cm}
\end{cases}
\end{equation}

%For  $p=q$ it is the classical eigenvalue problem for the $p$-Laplace operator which is widely studied by many authors (see for instance ...).

The spectrum of \eqref{pqe} and the first eigenvalue $\lambda^{\mathcal E}_{p,q}(\Omega)$, when $p\neq q$ and $m\equiv 1$, have been studied for instance in the case $p=2$ in \cite{brasco2019overview}, for any $p$ in \cite{garcia1987existence,otani1984certain, otani1988existence,franzina2010existence} and in \cite{drabek2011quasilinear} where the authors study also the weighted  case.
%Moreover, in \cite{kawohl2007simplicity} the author prove the simplicity for some quasilinear problem.
 It is known that $\lambda^{\mathcal E}_{p,q}(\Omega)$ is not simple, in general, for any $1<q<p^*$. Indeed in  \cite{idogawa1995first} the authors prove the simplicity for any $1<q\le p$,  while for $p<q<p^*$, $\lambda^{\mathcal E}_{p,q}(\Omega)$ could not be necessary simple. This fact has been observed for instance in \cite{brasco2019overview}, in the case $p=2$, and for any $p$ in \cite{nazarov2000one} and \cite{kawohl2007simplicity}, where the authors prove that the simplicity fails if $\Omega$ is a sufficiently thin spherical shell. 
%However there are domains $\Omega$ for which $\lambda^{\mathcal E}_{p,q}(\Omega)$ is simple even if $2<q<2^*$. %Up to our knowledge, the simplicity holds for any $1<q<2^*$ when $\Omega$ is a ball of $\R^n$(see \cite{erbetang}) and for any  convex and bounded set of the plane ( see \cite{lin1994uniqueness}). 

In this paper we study the main properties of $\lambda^{H}_{p,q}(\Omega)$ and of the corresponding eigenfunctions. %As we will see, analogous results hold also in the anisotropic case for
%\eqref{varcar}. \\
In particular our aim is to prove a reverse H\"older inequality for them.% eigenfunctions related to $\lambda_H(\Omega,m,p,q)$.

In the Euclidean case  if $p=q=2$ and $m\equiv 1$,  in \cite{Chiti1982, chiti1982reverse}, Chiti proved the following inequality for the first  eigenfunctions $v$ corresponding to the first Dirichlet eigenvalue of the Laplace operator  $\lambda^{\mathcal E}_2(\Omega)\equiv \lambda(\Omega)$
\begin{equation}\label{Chiti}
    \|v \|_{L^r(\Omega)} \le C(r,s,n,\lambda(\Omega))\| v\|_{L^s(\Omega)}, \quad 0< s<r,
\end{equation}
and  the equality case is achieved if and only if $\Omega$ is a ball. In \cite{alvino1998properties}, the authors prove \eqref{Chiti} for the first eigenfunctions of the $p$-Laplacian. Moreover, in \cite{alberico1999some}, the authors extend the result to the weighted case, and the inequality reads as follows
\begin{equation}\label{Chitiweig}
    \|v \|_{L^r(\Omega,m)} \le C(p,r,s,n,\lambda^{\mathcal E}_{p}(\Omega))\| v\|_{L^s(\Omega,m)}, \quad 0< s<r.
\end{equation}
 In the general case $p \neq q$, in the Euclidean case, a Chiti type inequality is proved in the case $m\equiv 1$ in \cite{carroll2015reverse}  and \cite{carroll2017rate} when $p=2$ and for any $p$ respectively. More precisely in \cite{carroll2017rate} the authors prove the following inequality 
\begin{equation}\label{Chitipq1}
    \|v \|_{L^s(\Omega)} \le C(p,q,s,n,\lambda^{\mathcal E}_{p,q}(\Omega))\| v\|_{L^q(\Omega)}, \quad q<s.
\end{equation}

The goal of this paper is to prove a   Chiti type inequality in the spirit of \eqref{Chitiweig} and \eqref{Chitipq1} for the first eigenfunctions of the general weighted eigenvalue problem \eqref{APQ}. In particular our main theorem is the following
\begin{thm}\label{chitiAPQ}
    Let $\Omega\subset\mathbb R^n$ be an open, bounded and connected set. Let $1< q\le p$, and let $u$ be an  eigenfunction corresponding to the first eigenvalue \eqref{varcar}. Then the following statements hold
    \begin{itemize}
    \item[i)]
    There exists a  constant $C=C(n,p,q,r,\lambda^{H}_{p,q}(\Omega))$ such that
  \begin{equation} \label{q1q2comparison}
       \|u \|_{L^r(\Omega,m)} \le C\,\,\| u\|_{L^q(\Omega,m)}, \quad q\le r;
    \end{equation}
    \item[ii)] There exists a  constant $C=C(n,p,q,r,\lambda^{H}_{p,q}(\Omega))$ such that
    \begin{equation}
    \label{caso_infinito}
    \|u\|_{L^\infty(\Omega)} \le C \|u\|_{L^r(\Omega,m)} \quad 1\le r<\infty.
    \end{equation}
    \end{itemize}
The equality cases hold if and only if $\Omega$ is a Wulff shape.
\end{thm}%Moreover in \cite{kawohl2007simplicity} it is 
We stress that this result gives, in particular, a Chiti type inequality for the eigenfunctions corresponding to the first weighted eigenvalue of the anisotropic $p$-Laplacian and extend \eqref{Chitipq1} to the weighted case. The proof is based on symmetrization techniques and a comparison between the eigenfunctions corresponding to the first eigenvalue \eqref{varcar} and the first eigenfunctions of a suitable symmetrical eigenvalue problem. 
\\ 

The structure of the paper is the following.
In Section 2 we fix some notation, we  recall some basic properties of the Finsler norms and we give a brief overview about convex symmetrization. In Section 3 we study the main properties of $\lambda^H_{p,q}(\Omega)$  and  a Faber-Krahn type inequality. In the last Section we prove Theorem \ref{chitiAPQ} by using symmetrization arguments.

\section{Notations and preliminaries}
Throughout this article,  $|\cdot|$ denotes  the Euclidean norm in $\mathbb{R}^n$, while $\cdot$ is  the standard Euclidean scalar product for  $n\geq2$. Moreover we denote by $|\Omega|$ the Lebesgue measure  of $\Omega\subseteq \mathbb{R}^n$, $B_R$ the Euclidean ball centered at the origin with radius $R$ and by $\omega_n$ the measure of the unit ball.

\noindent Let $E\subseteq\mathbb{R}^n$ be a bounded, open set and let $\Omega\subseteq\R^{n}$ be a measurable set.  We recall now the definition of the perimeter of $\Omega$ in $E$ in the sense of De Giorgi, that is
\begin{equation*}\label{def_per_fusco}
P(\Omega; E)=\sup\left\{  \int_\Omega {\rm div} \varphi\:dx :\;\varphi\in C^{\infty}_c(E;\mathbb{R}^n),\;||\varphi||_{\infty}\leq 1 \right\}.
\end{equation*}
The perimeter of $\Omega$ in $\mathbb{R}^n$ will be denoted by $P(\Omega)$ and, if $P(\Omega)<\infty$, we say that $E$ is a set of finite perimeter. Some  references for  results relative to the sets of finite perimeter are for example \cite{maggi2012sets}.  Moreover, if $\Omega$ has Lipschitz boundary, we have that
\begin{equation*}
P(\Omega)=\mathcal{H}^{n-1}(\partial \Omega). 
\end{equation*}

%\begin{thm}[Coarea formula]
%	Let $f:\Omega\to\R$ be a Lipschitz function, and $u:\Omega\to\R$ a measurable function. Then
%	\begin{equation}
%	\label{coarea}
%	{\displaystyle \int _{\Omega}u|\nabla f(x)|dx=\int _{\mathbb {R} }dt\int_{(\Omega\cap f^{-1}(t))}u(y)\, d\mathcal {H}^{n-1}(y)}
%	\end{equation}
%\end{thm}
%Throughout this paper we will denote by $B_R= \{x\in \mathbb{R}^n : \|x\| < R\}$ the ball centered at the origin with radius $R>0$, where $\|\cdot \|$ is the classical euclidean distance; by $\Omega$ a bounded $C^{2,\alpha}$ and simply connected domain with finite Lebeasgue measure, where $C^{2,\alpha}$ stands for the $\alpha$-H\"olderian space. We denote by $\mathcal{H}^{n-1}$ the $(n-1)$-dimensional Hausdorff measure in $\mathbb{R}^n$ and by $|\cdot|$ the Lebeasgue measure in $\mathbb{R}^n$.\\
\subsection{The anisotropic norm}
Let $H:\mathbb{R}^n \longrightarrow [0,+\infty]$, $n\ge 2$, be a $C^2(\mathbb{R}^n \setminus \{0\})$ convex function which is $1$-homogeneous, i.e.
\begin{equation} \label{HP}
H(t\xi)= |t| H(\xi) \,\,\,\,\,\, \forall \xi \in \mathbb{R}^n \, , \,\, \forall t \in \mathbb{R},
\end{equation}
and such that
\begin{equation} \label{control}
\gamma |\xi | \le H(\xi) \le \delta |\xi|,
\end{equation}
for some positive constants $\gamma \le \delta$. \\
These properties guarantee that $H$ is a norm in $\mathbb{R}^n$. Indeed by \eqref{control} we have that $H(\xi)=0$ if and only if $\xi = 0$. It is homogeneous by \eqref{HP} and the triangular inequality is a consequence of the convexity of the function $H$: if $\xi,\eta \in \mathbb{R}^n$, then
\begin{equation*}
	\frac{H(x+y)}{2}= H\bigg(\frac{x}{2}+\frac{y}{2}\bigg)\le \frac{H(x)}{2}+\frac{H(y)}{2}.
\end{equation*}
Because of \eqref{HP}, we can assume that the set
\begin{equation*}
K = \{\xi \in \mathbb{R}^n : H(\xi)\le 1 \}
\end{equation*}
is such that $|K|=\omega_n$, where $\omega_n$ is the measure of the unit sphere in $\mathbb{R}^n$. We can define the support function of $K$ as
\begin{equation}
H^{\circ}(x)= \sup_{\xi \in K} \left< x , \xi \right>,
\end{equation} 
where $\left< \cdot , \cdot \right>$ denotes the scalar product in $\mathbb{R}^n$. $H^{\circ}:\mathbb{R}^n \longrightarrow [0,+\infty]$ is a convex, homogeneous function in the sense of \eqref{HP}. Moreover $H$ and $H^{\circ}$ are polar to each other, in the sense that
\begin{equation*}
H^{\circ}(x) = \sup_{\xi \neq 0 } \frac{\left< x , \xi \right>}{H(\xi)}
\end{equation*}
and
\begin{equation*}
H(x) = \sup_{\xi \neq 0 } \frac{\left< x , \xi \right>}{H^{\circ}(\xi)}.
\end{equation*}
$H$ is the support function of the set
\begin{equation*}
K^{\circ} = \{x \in \mathbb{R}^n : H^{\circ}(x)\le 1 \}.
\end{equation*}
The set $\mathcal{W}=\{x \in \mathbb{R}^n : H^{\circ}(x)< 1 \}$ is the so-called Wulff shape centered at the origin. We set $k_n = |\mathcal{W}|$. More generally we will denote by $\mathcal{W}_R(x_0)$ the Wulff shape centered in $x_0 \in \mathbb{R}^n$ the set $R\mathcal{W}+x_0$, and $\mathcal{W}_R(0)=\mathcal{W}_R$.\\
The following properties hold for $H$ and $H^{\circ}$:
\begin{equation} \label{prodsc}
H_{\xi}(\xi)\cdot \xi = H(\xi), \;\;\;\; H^{\circ}_{\xi}(\xi)\cdot \xi = H^{\circ}(\xi),
\end{equation} 
\begin{equation} \label{HHcirc}
H(H^{\circ}_{\xi}(\xi))= H^{\circ}(H_{\xi}(\xi))=1 \,\,\,\,\, \forall \xi \in \mathbb{R}^n \setminus \{0\},
\end{equation} 
\begin{equation}\label{Hcircxi}
H^{\circ}(\xi)H_{\xi}(H^{\circ}_{\xi}(\xi))=H(\xi)H^{\circ}_{\xi}(H_{\xi}(\xi)) = \xi \,\,\, \forall \xi \in \mathbb{R}^n \setminus \{0\}.
\end{equation}
If $\Omega \subset \mathbb{R}^n$ is an open bounded set with Lipschitz boundary and $E$ is an open subset of $\mathbb{R}^n$, we can give a generalized definition of perimeter of $E$ with respect to the anisotropic norm as follows (see for instance \cite{Amar1994anotion})
\begin{equation*}
	P_H(E,\Omega)= \int_{\partial^* E \cap \Omega} H(\nu )\,d\mathcal{H}^{n-1},
\end{equation*}
where $\partial^*E$ is the reduced boundary of $E$  (for the definition see \cite{evans2015measure}), $\nu$ is its Euclidean outer normal and $\mathcal{H}^{n-1}$ is the $(n-1)-$dimensional Hausdorff measure in $\mathbb{R}^n$. Clearly, if $E$ is open, bounded and Lipschitz, then the outer unit normal exists almost everywhere and
\begin{equation}
P_H(E,\mathbb{R}^n):=P_H (E) = \int_{\partial E} H(\nu )\,d\mathcal{H}^{n-1}.
\end{equation}
 %We will denote $P_H(E)=P_H(E,\mathbb{R}^n)$. 
By \eqref{control} we have that
\begin{equation*}
\gamma P(E) \le P_H(E)\le \delta P(E).
\end{equation*}
In \cite{alvino1990comparison} it is shown that if $u\in W^{1,1}(\Omega)$, then for a.e. $t>0$ 
\begin{equation} \label{AFLTAB}
	-\frac{d}{dt}\int_{\{u>t\}} H(\nabla u)\,dx = P_H(\{u>t\},\Omega) = \int_{\partial^*\{u>t\}\cap \Omega} \frac{H(\nabla u)}{|\nabla u|}\,d\mathcal{H}^{n-1}.
\end{equation}
Moreover an isoperimetric inequality for the anisotropic perimeter holds (for instance see \cite{busemann1949isoperimetric, dacorogna1992wulff, trombetti1997convex, fonseca1991uniqueness})
\begin{equation} \label{IIA}
P_H( E ) \ge n k^{\frac{1}{n}}_n |E|^{1-\frac{1}{n}}.
\end{equation}

%Eventually we give the definition of anisotropic distance (see \cite{crasta2007distance}). Let $\Omega \subset \mathbb{R}^n$ be an open, bounded and with Lipschitz boundary. We will call anisotropic distance and denote it by $d_H$ the following
%\begin{equation} \label{anisdist}
%    d_H(x) = \inf_{y\in \partial \Omega} H^{\circ}(x-y).
%\end{equation}
%By property \eqref{HHcirc} it is straightforward to check that
%\begin{equation*}
%    H(\nabla d_H(x))=1.
%\end{equation*}
%Let $t>0$, we will denote by 
%\begin{equation*}
 %   \Omega_t = \{x\in\Omega : d_H(x)>t\}
%\end{equation*}

\subsection{Convex symmetrization}
Let $\Omega \subset \R^n$ be an open, bounded, connected set.
Let $f: \Omega \longrightarrow [0,+\infty]$ be a measurable function. The decreasing rearrangement $f^*$ of $f$ is defined as follows
\begin{equation*}
f^*(s)= \inf \{t\ge 0 :  \mu(t)<s\} \,\,\,\,\,\, s\in [0,|\Omega|],
\end{equation*}
where 
\begin{equation*}
    \mu(t) = |\{x\in \Omega: |f(x)|>t\}|,
\end{equation*}
is the distribution function of $f$. We recall that the Schwarz symmetrand of $f$ is a radially spherically function  defined as follows
\[
f^{\sharp}(x) = f^*(\omega_n |x|^n) \,\,\,\,\,\,\,\, x\in \Omega^{\sharp}.
\]
where $\Omega^{\sharp}$ is the ball centered at the origin such that $|\Omega^{\sharp}|=|\Omega|$.
The convex symmetrization $f^{\star}$ of $f$ instead, is a  function symmetric with respect to $H^{\circ}$ defined as follows
\begin{equation*}
f^{\star}(x) = f^*(k_n (H^{\circ}(x))^n) \,\,\,\,\,\,\,\, x\in \Omega^{\star},
\end{equation*}
where $\Omega^\star$ is a Wulff shape centered at the origin and such that $|\Omega^{*}|=|\Omega|$ (see \cite{trombetti1997convex}).
We stress that both $f^{\star}$ and $f^{\sharp}$ are defined by means the decreasing rearrangement $f^*$ but they have different symmetry. 
In particular it is well known that the functions $f$, $f^*$, $f^{\sharp}$  and $f^{\star}$ are equimeasurable, i.e.
\begin{equation*}
|\{f>t\}| =|\{f^{\sharp} > t\}| = |\{f^* > t\}| = |\{f^{\star}>t\}| \,\,\,\,\, t\ge 0.
\end{equation*}
As a consequence, if $f\in L^p(\Omega)$, $p\ge 1$, then 
\begin{equation}\label{conservationnorms}
\| f\|_{L^p(\Omega)}\| =\| f^{\sharp} \|_{L^p(\Omega^{\sharp})} = \| f^* \|_{L^p([0,|\Omega|])} = \|f^{\star}\|_{L^p(\Omega^{\star})}.
\end{equation}
Regarding the norm of the gradient, a generalized version of the  well known P\'olya-Szeg\"o inequality holds and it states (see for instance \cite{trombetti1997convex})
\begin{thm}{(P\'olya-Szeg\"o principle)}\label{polyaszego}
If $w \in W^{1,p}_0(\Omega)$ for $p\ge 1$, then we have that
\begin{equation*}
    \int_{\Omega}H(\nabla u)^p\,dx \ge \int_{\Omega^{\star}} H(\nabla u^{\star})^p\,dx.
\end{equation*}
where $\Omega^\star$ is the Wulff Shape such that $|\Omega^\star|=|\Omega|$.
\end{thm}
The equality case, proved in the euclidean case by \cite{ziemer1988minimal}, was studied in \cite{ferone2004convex}. We will state it for sake of completeness
\begin{thm}\label{BZanisotropic}
    Let $u$ be a non-negative function in $W^{1,p}(\mathbb{R}^n)$, for $1<p<+\infty$, such that
    \begin{equation*}
        |\{|\nabla u^\star|=0\}\cap \{0<u^\star<\esssup u\}|=0.
    \end{equation*}
    Then
    \begin{equation*}
        \int_{\mathbb{R}^n}H(\nabla u)^p\,dx = \int_{\mathbb{R}^n} H(\nabla u^{\star})^p\,dx
    \end{equation*}
    if and only if $u=u^\star$ a.e in $\mathbb{R}^n$, up to translations.
\end{thm}
Clearly Theorem \ref{BZanisotropic} can be adapted in the case of a $W^{1,p}_0(\Omega)$ function.\\ 
We conclude this section by recalling some known properties about rearrangements that we will use in the proof of  the main theorem.
The following result is the well- known  Hardy-Littlewood inequality (see \cite{kesavan2006symmetrization})
\begin{equation}
\int_{\Omega} |f(x)g(x)| \,dx \le \int_0^{|\Omega|} f^*(s)g^*(s)\,ds.
\end{equation}
So, if we consider $g$ as the characteristic function of the set $\{x\in\Omega : u(x)>t\}$, for some measurable function $u:\Omega \rightarrow \mathbb{R}$ and $t\ge 0$, then we get
\begin{equation}\label{fleqf*}
\int_{\{u>t\}} f(x)\,dx \le \int_0^{\mu(t)} f^*(s)\,ds,
\end{equation}
where, again, $\mu(t)$ is the distribution function of $u$.
Finally we recall the definition of dominated rearrangements (see for instance \cite{alt} and \cite{chong-rice}).
\begin{defn}
    Let $f,g \in L^1(\Omega)$ be nonnegative functions. We say that $g$ is dominated by $f$ and we write $g\prec f$ if the following two statements hold
\begin{item}
\item[i)] $\displaystyle\int_{0}^s g^* (t)\,dt \le \displaystyle\int_{0}^s f^* (t)\,dt $; \\
\item[ii)]$\displaystyle\int_{0}^{|\Omega|} g^* (t)\,dt =\displaystyle\int_{0}^{|\Omega|} f^* (t)\,dt $.
\end{item}
\end{defn}
In \cite{alt} the authors prove the following result
\begin{prop}
    \label{domi}
Let $f,g,h$ be positive and such that  $hf, hg\in L^1(\Omega)$. Let $F$ be a convex, nonnegative function such that $F(0)=0$. If $hg \prec hf$ Then 
\[
\int_0^{|\Omega|}h^* F(g^*) \, dt \le \int_0^{|\Omega|}h^* F(f^*) \, dt.
\]
Moreover, if $F$ is strictly convex the equality holds if and only if $f^*\equiv g^*$ a.e. in $[0,|\Omega|]$.     
\end{prop}

%In particular \eqref{conservationnorms} and Theorem \ref{polyaszego} allows us to prove the following anisotropic Faber-Krahn inequality for $\lambda_{p,q}(\Omega)$ and Theorem \ref{BZanisotropic} its related equality case.

\section{The (p,q)-anisotropic Laplacian}
In this section we study the main properties of \eqref{varcar} and of the corresponding minimizers.
Let $\Omega\subset \R^n$, $n\ge 2$  be an open, bounded and connected set. Let $m\in L^{\infty}(\Omega)$ be a positive function and  $p, q$ be such that $1<p<\infty$, and  $1<q<p^*$, where $p^*=np/(n-p)$, if $p<n$, and $p^*=\infty$, if $p\ge n$. A function $v\in W_0^{1,p}(\Omega)$ is a weak solution to the problem \eqref{APQ} corresponding to $\lambda$ if
\begin{equation}
\label{ws}
\int_{\Omega}(H(\nabla v))^{p-1}H_{\xi}(\nabla v) \cdot \nabla \varphi \, dx=\lambda \|v\|_{q,m}^{p-q}\int_{\Omega}m(x)\,|v|^{q-2}\,v \,\varphi\,dx,
\end{equation}
for every $\varphi \in W_0^{1,p}(\Omega)$.
By standard argument of calculus of variations it is not difficult to prove the following result
\begin{thm}\label{es-cv-segno}
Let $n\ge 2$ and $\Omega\subset \R^n$, be an open, bounded and connected set and let $p,q$ and $m$ be as above. Then $\lambda^H_{p,q}(\Omega)$, defined in \eqref{varcar}, is strictly positive and it is actually a minimum. Moreover any minimizer  is a weak solution to the problem \eqref{APQ}, with $\lambda=\lambda^H_{p,q}(\Omega)$ and it has constant sign.
\end{thm}
As regard the simplicity,  we have
\begin{thm}
\label{simpl}
Let $n\ge 2$ and $\Omega\subset \R^n$, be an open, bounded and connected set and let $p$ and $m$ be as above and let $1<q\le p$. Then $\lambda^H_{p,q}(\Omega)$ is simple, that is there exists a unique corresponding eigenfunction up to multiplicative constants. 
\end{thm}
The proof of the previous result is contained in \cite{kawohl2007simplicity}, where the authors consider a more general class of quasilinear operators.
Finally  we have the following 
\begin{thm}
\label{car-seg}
Let $n\ge 2$ and let $\Omega\subset \R^n$ be an open, bounded and connected set. Let $p$ and $m$ be as above and let $1<q\le p$. Any nonnegative  function $v \in W_0^{1,p}(\Omega)$ ,  which is a weak solution to the problem \eqref{APQ}, for some $\lambda>0$, is a first eigenfunction, that is $\lambda =\lambda^H_{p,q}(\Omega)$.
\end{thm}
\begin{proof}
The proof follows by standard arguments and a general Picone inequality contained in \cite{brasco2014convexity}.

Let $v$ be a non-negative weak solution to the problem  \eqref{APQ} corresponding to $\lambda$. By the strong maximum principle we have that $v>0$ in $\Omega.$ Let  $u$ be the first positive  eigenfunction corresponding to $\lambda^H_{p,q}(\Omega)$ such that 
\begin{equation}
\label{norm}
\|u\|_{L^q(\Omega,m)}=\|v\|_{L^q(\Omega,m)}.
\end{equation}
Then 
\begin{equation}
\label{f-eig}
\int_{\Omega}(H(\nabla u))^p\,dx=\lambda^H_{p,q}(\Omega)\left(\int_{\Omega}m(x) \,u^q \,dx \right)^{\frac pq}.
\end{equation}
Being $v$ a weak positive solution to \eqref{APQ} corresponding to $\lambda$, we can chose  $\varphi =\displaystyle \frac{u^q}{v^{q-1}}$ as test function in \eqref{ws} obtaining
\begin{multline}
 \label{wsv}   
 \int_{\Omega}\left((H(\nabla v)\right)^{p-1}H_{\xi}(\nabla v )  \cdot \nabla \left( \frac{u^q}{v^{q-1}}\right)\,dx 
=\lambda \|m^{\frac{1}{q}} v\|_q^{p-q}\int_{\Omega}m(x)\,u^q\,dx \\
= \lambda \left(\int_{\Omega}m(x) \,u^q \,dx \right)^{\frac pq},
%-\\ - (q-1)\int_{\Omega}\left(\frac{u}{v}\right)^{q}\left((H(\nabla v)\right)^{p-1}H_{\xi}(\nabla v)  \cdot \nabla v\,dx  \\
%q\int_{\Omega}\left(\frac{u}{v}\right)^{q-1}\left((H(\nabla v)\right)^{p-1}H_{\xi}(\nabla v)  \cdot \nabla u\,dx -\\ - (q-1)\int_{\Omega}\left(\frac{u}{v}\right)^{q}\left((H(\nabla v)\right)^{p}\,dx
%\\
%= \lambda \left(\int_{\Omega}m(x) \,u^q \,dx \right)^{\frac pq}
\end{multline}
where last equality follows by \eqref{norm}.
In the left-hand side, we can apply the general Picone inequality (see Proposition 2.9 in \cite{brasco2014convexity}) and we have 
\begin{equation*}
\int_{\Omega}\left(H(\nabla v)\right)^q \left(H(\nabla u)\right)^{p-q} \,dx \ge \lambda \left(\int_{\Omega}m(x) \,u^q \,dx \right)^{\frac pq}.
\end{equation*}
By the H\"older inequality, the normalization \eqref{norm} and \eqref{f-eig} we get that 
$\lambda^{H}_{p,q}(\Omega)\ge \lambda$, that implies $u=v$.
\end{proof}
\subsection{The case $\Omega= \mathcal{W}_R$}
In  this subsection we study the problem \eqref{APQ} when $\Omega$ is a Wulff shape. In this case the eigenfunctions inherit some symmetry properties.
Let be $\Omega= \mathcal{W}_R$ and let $m\in L^{\infty}(\mathcal{W}_R)$ be a positive function such that $m(x)=m^{\star}(x)$. Then problem \eqref{APQ} becomes
\begin{equation} \label{ARPQ}
\begin{cases}
-\mathcal{L}_{p}(v)=\lambda m^{\star}(x) \| v\|_{q, m^\star}^{p-q}|v|^{q-2}v & \mbox{in}\ \mathcal{W}_R\vspace{0.2cm}\\
v=0&\mbox{on}\ \partial \mathcal{W}_R\vspace{0.2cm}.
\end{cases}
\end{equation}
The following result holds
\begin{prop}
\label{rad}
Let $1<p<\infty$ and $1<q\le p$. Let $v \in C^1(\overline{\Omega})\cap C^{1,\alpha}(\Omega)$ be a first  positive  eigenfunction to the problem \eqref{ARPQ}. Then there exists a decreasing function $\rho(r)$, $r \in [0,R]$, such that $\rho \in C^{\infty}((0,R))\cap C^1([0,R])$, $\rho'(0)=0$ and $v(x)=\rho(H^o(x))$.
\end{prop}
\begin{proof}
By the simplicity we can assume that  $\|v\|_{L^q(\mathcal W_R,m^{\star})}=1$. 
 Let $B_R$ be the ball centered at the origin with radius $R>0$,  and let us consider the weighted p-Laplace eigenvalue problem in $B_R$
\begin{equation}\label{euclideanradial}
    \begin{cases}
-\Delta_p z=\lambda \tilde m(|x|)\|z\|_{q,m^{\star}}^{p-q}|z|^{q-2}z & \mbox{in}\ B_R\vspace{0.2cm}\\
z=0&\mbox{on}\ \partial B_R,\vspace{0.2cm}
\end{cases}
\end{equation}
%By the simplicity we can assume that  $\|v\|_{L^q(\mathcal W_R, m)}=1$.
where $\tilde m(r)=m^*(k_n r^n)$, $0 \le r\le R$.
Let $z$ be the positive eigenfunction corresponding to the first eigenvalue $\lambda^{\mathcal E}_{p,q}(B_R)$ to the problem \eqref{euclideanradial}, such that $\|z\|_{L^q(B_R,\tilde{m})}=\|v\|_{L^q(\mathcal W_R,m^{\star})}=1$.
Then uniqueness guarantees that $z$ is  radially symmetric, which means that there exists a positive one dimensional function $\rho_p:r\in[0,R]\to \mathbb{R}^+$
such that $z(x)=\rho_p(|x|)$, and $\rho_p $ solves the following problem
\begin{equation}\label{radialeucl}
\begin{cases}
    -(p-1)|\rho'_p|^{p-2}\rho''_p+\frac{n-1}{r}|\rho'_p|^{p-1}= \lambda^{\mathcal E}_{p,q}(B_R)\tilde m|\rho_p|^{q-2}\rho_p,& r \in (0,R) \\
\rho'_p(0)=\rho_p(R)=0.
\end{cases}
\end{equation}
In particular, integrating equation \eqref{radialeucl}, it is possible to see that $\rho'_p$ is zero only when $r=0$ and consequently that $\rho_p$ is strictly decreasing in $[0,R]$.
Now we can come back to the anisotropy. Indeed if we consider $w=\rho_p(H^\circ(x))$ then using properties \eqref{prodsc}-\eqref{Hcircxi} and the regularity of $H$, by construction, we obtain that $w(x)$ is a solution to problem \eqref{ARPQ}, which is positive and radial with respect to the anisotropic norm. The simplicity and Theorem \ref{car-seg} imply that $v=w$ and this concludes the proof.
\end{proof}
\begin{rem}
We stress that the proof of the previous result shows that the first eigenvalue $\lambda^H_{p,q}(\mathcal W_R)$ coincides with the first eigenvalue of problem \eqref{euclideanradial}. %of the weighted $(p,q)$-Laplacian $\lambda^{\mathcal E}_{p,q}(\mathcal B_R)$ of the Euclidean ball having the same radius $R$.
\end{rem}

\subsection{A Faber-Krahn type inequality}
\begin{thm}
Let $\Omega\in \mathbb{R}^n$, $n\ge 2$, be an open, bounded and connected set and let $1<q\le p$. Then 
\begin{equation} \label{FK}
    \lambda^H_{p,q}(\Omega)\ge \lambda^H_{p,q}(\Omega^{\star}),
\end{equation}
where $\Omega^\star$ is the Wulff shape such that $|\Omega^\star|=|\Omega|$. The equality case holds if and only if $\Omega=\Omega^{\star}$ and $m=m^{\star}$ a.e. in $\Omega$,  up to translations, where $m^{\star}$ is the convex symmetrization of $m$.
\end{thm}
\begin{proof}
We observe that $\lambda^H_{p,q}(\Omega^{\star})$ has the following variational characterization
\begin{equation}\label{varcarrad}
    \lambda^H_{p,q}(\Omega^{\star}) = \inf_{\substack{w\in W^{1,p}_0(\Omega^{\star}), \\ w \neq 0}}\frac{\ds\int_{\Omega^{\star}} H(\nabla w)^p\,dx}{\bigg(\ds\int_{\Omega^{\star}} m^{\star}|w|^q\, dx\bigg)^{\frac{p}{q}}}.
\end{equation}
The Faber-Krahn inequality is a straightforward application of the P\'olya-Szeg\"o principle and the Hardy-Littlewood inequality. Indeed if $u$ is a positive eigenfunction corresponding to $\lambda^H_{p,q}(\Omega)$, then
\begin{equation}\label{fabkr}
    \lambda^H_{p,q}(\Omega) = \frac{\ds\int_\Omega H(\nabla u)^p\,dx}{\bigg(\ds\int_\Omega m{u}^q\,dx\bigg)^{\frac{p}{q}}} \ge \frac{\ds\int_{\Omega^{\star}} H(\nabla u^{\star})^p\,dx}{\bigg(\ds\int_{\Omega^{\star}}m^{\star}(u^{\star})^q\,dx\bigg)^{\frac{p}{q}}}\ge \lambda^H_{p,q}(\Omega^{\star}).
\end{equation}
Let us now consider the equality case. From \eqref{fabkr}, P\'olya-Szeg\"o principle and Hardy-Littlewood inequality we get
\begin{equation*}
   1\le\frac{\ds\int_\Omega H(\nabla u)^p\,dx}{\ds\int_{\Omega^{\star}} H(\nabla u^{\star})^p\,dx}= \frac{\bigg(\ds\int_{\Omega^{\star}} m^{\star}(u^{\star})^q\,dx\bigg)^{\frac{p}{q}}}{\bigg(\ds\int_\Omega m u^q\,dx\bigg)^{\frac{p}{q}}}\le 1.
\end{equation*}
It follows that
\begin{equation}\label{PSeq}
    \int_\Omega H(\nabla u)^p\,dx\,dx=\int_{\Omega^{\star}} H(\nabla u^{\star})^p\,dx,
\end{equation}
and
\begin{equation}\label{HLeq}
    \bigg(\ds\int_\Omega m u^q\,dx\bigg)^{\frac{p}{q}}=\bigg(\ds\int_{\Omega^{\star}} m^{\star}(u^{\star})^q\,dx\bigg)^{\frac{p}{q}}.
\end{equation}
The thesis follows from \eqref{PSeq}, Theorem \ref{BZanisotropic} and \eqref{HLeq}.
\end{proof}
\section{A Chiti type inequality}
In this section we prove a reverse H\"older inequality for the eigenfunctions corresponding to $\lambda^H_{p,q}(\Omega)$. We first prove the following proposition as in the spirit of the Talenti result  contained in \cite{talenti1976elliptic} (see also \cite{alberico1999some, trombetti1997convex, alvino1998properties, amato2023talenti, sannipoli2022comparison}).
\begin{prop}
Let $\Omega\subset\mathbb{R}^n$, $n\ge 2$, be an open, bounded and connected set, $1<q\le p$, and let $m\in L^{\infty}(\Omega)$ be a positive function.  Let $u$ be a positive eigenfunction corresponding to $\lambda^H_{p,q}(\Omega)$. 
Then we have that
\begin{equation}\label{diffinequality}
    (-{u^*}'(s))^{p-1} \le n^{-p} k_n^{-\frac{p}{n}}\lambda^H_{p,q}(\Omega) \| u\|_{q,m}^{p-q}\int_0^{s} m^*(u^*)^{q-1}(r)\,dr, \qquad s\in [0,|\Omega|].
\end{equation}
In particular the equality case holds if and only if $\Omega = \Omega^{\star}$ and $m=m^{\star}$ a.e. in $\Omega$,  up to translations, where $m^{\star}$ is the convex symmetrization of $m$.
\end{prop}
\begin{proof}
The function $u$ is a weak solution to the problem 
\begin{equation*} 
\begin{cases}\label{solprob}
-\mathcal{L}_{p}(u)=\lambda^H_{p,q}(\Omega) m\| u\|_{q,m}^{p-q}u^{q-1} & \mbox{in}\ \Omega\vspace{0.2cm}\\
u=0&\mbox{on}\ \partial \Omega\vspace{0.2cm}.
\end{cases}
\end{equation*}
 Let us integrate  both sides of the equation in \eqref{solprob} on the super level sets of $u$. Then integrating by parts we get
\begin{equation*}
    -\int_{\{u=t\}} H(\nabla u)^{p-1}H_{\xi}(\nabla u)\cdot \nu\,d\mathcal{H}^{n-1}=\lambda^H_{p,q}(\Omega)\,\|u\|_{q,m}^{p-q}\int_{\{u>t\}}  mu^{q-1} \,dx.
\end{equation*}
Since on $\{u=t\}$ we have that $\nu = -\nabla u /|\nabla u|$ and using property \eqref{prodsc}, we have that the left hand side of the previous equation becomes
\begin{equation*}
    -\int_{\{u=t\}} H(\nabla u)^{p-1}H_{\xi}(\nabla u)\cdot \nu\,d\mathcal{H}^{n-1} = \int_{\{u=t\}} \frac{H(\nabla u)^{p}}{|\nabla u|}\,d\mathcal{H}^{n-1}.
\end{equation*}
By H\"older inequality and since $H$ is 1-homogeneous, we get
\begin{equation*}
    \int_{\{u=t\}} \frac{H(\nabla u)}{|\nabla u| }\,d\mathcal{H}^{n-1}\le \bigg(\int_{\{u=t\}} \frac{H(\nabla u)^p}{|\nabla u| }\,d\mathcal{H}^{n-1}\bigg)^{\frac{1}{p}}\bigg(\int_{\{u=t\}} \frac{1}{|\nabla u|}\,d\mathcal{H}^{n-1}\bigg)^{1-\frac{1}{p}}
\end{equation*}
Since 
\begin{equation*}
    \int_{\{u=t\}} \frac{H(\nabla u)}{|\nabla u| }\,d\mathcal{H}^{n-1}=-\frac{d}{dt}\int_{\{u>t\}} H(\nabla u)\,d\mathcal{H}^{n-1} = P_H(\{u>t\}),
\end{equation*}
and
\begin{equation*}
    -\mu'(t)= \int_{\{u=t\}} \frac{1}{|\nabla u|}\,d\mathcal{H}^{n-1},
\end{equation*}
applying the isoperimetric inequality \eqref{IIA} we have
\begin{equation*}
    (-\mu'(t))^{1-p}\int_{\{u=t\}} \frac{H(\nabla u)^p}{|\nabla u| }\,d\mathcal{H}^{n-1} \ge n^pk_n^{\frac{p}{n}}\mu(t)^{p-\frac{p}{n}}
\end{equation*}
Since $\mu'(t)=\frac{1}{{u^*}'(\mu(t))}$ we have
\begin{equation*}
    (-{u^*}'(\mu(t)))^{p-1} \le n^{-p}k_n^{-\frac{p}{n}} \lambda^H_{p,q}(\Omega)\| u\|_{q,m}^{p-q}\mu(t)^{\frac{p}{n}-p}\int_{\{u>t\}}  m(x)u^{q-1} \,d\mathcal{H}^{n-1}.
\end{equation*}
Using \eqref{fleqf*} and calling $s=\mu(t)$ we have
\begin{equation*}
    (-{u^*}'(s))^{p-1} \le n^{-p}k_n^{-\frac{p}{n}}\lambda^H_{p,q}(\Omega)\| u\|_{q,m}^{p-q}\,s^{\frac{p}{n}-p}\int_0^{s} m^* (u^*)^{q-1} \,dr.
\end{equation*}
An application of the Hardy-Littlewood inequality gives the desired result.
\end{proof}
%We stress the fact that a quantitative version of the Talenti inequality has been recently proved in \cite{amato2023talenti}.\\ \\
The main tool we use in order to prove Theorem \ref{chitiAPQ} is a suitable comparison result between $u$ and an eigenfunction $z$ of a suitable eigenvalue problem.
More precisely, let $\Omega^{\star}_\lambda$ be the Wulff shape centered at the origin such that $\lambda^H_{p,q}(\Omega)$ is the first eigenvalue to the following symmetric problem 
\begin{equation}
\label{pbrad}
    \begin{cases}
       -\mathcal{L}_{p,q}(z)=\mu m^\star \| z\|_{q,m^\star}^{p-q}z^{q-1} & \mbox{in}\ \Omega^{\star}_\lambda\vspace{0.2cm}\\
z=0&\mbox{on}\ \partial \Omega^{\star}_\lambda\vspace{0.2cm}, \end{cases}
\end{equation}\\
We stress that the Faber-krahn inequality \eqref{FK} implies that 
\begin{equation}\label{OgeW}
    |\Omega|\ge |\Omega^{\star}_\lambda|,
\end{equation}
and hence $m^\star$ is well defined in $\Omega^\star_\lambda$.

Moreover if $z$ is a positive eigenfunction corresponding to $\lambda^H_{p,q}(\Omega)$, we observe that repeating the same argument as before, by Proposition \ref{rad}, for any $1<q\le p$ we have
\begin{equation}
\label{talenti-rad}
(-{z^*}'(s))^{p-1} = n^{-p} k_n^{-\frac{p}{n}}\lambda^H_{p,q}(\Omega) \| z^*\|_{q,m^*}^{p-q}\int_0^{s} m^*(z^*)^{q-1}(r)\,dr.
\end{equation}

The following proposition gives a  comparison result between the eigenfunctions $u$ and $z$ when they are normalized with respect the $L^{\infty}$ norm.
\begin{prop} \label{proppointwisecomp}
Let $\Omega\subset\mathbb{R}^n$, $n\ge 2$, be an open, bounded and connected set, $1<q\le p$ and let $m\in L^{\infty}(\Omega)$ be a positive function.  Let $u$ be a positive solution to the problem \eqref{APQ} corresponding to  $\lambda^H_{p,q}(\Omega)$ and let  $z$ be a positive eigenfunction to the problem \eqref{pbrad} corresponding to $\lambda^H_{p,q}(\Omega)$ such that  
 \[\|u\|_{L^{\infty}(\Omega)}=\|z\|_{L^{\infty}(\Omega^\star_\lambda)}\]
 Then 
\begin{equation*}
    u^*(s)\ge z^*(s), \qquad \forall s \in [0,|\Omega^\star_\lambda|],
\end{equation*}
 where $u^*$ and $z^*$ are respectively the decreasing rearrangements of $u$ and $z$.
 The equality case holds if and only if $\Omega=\Omega^{\star}_\lambda$ and $m=m^{\star}$ a.e. in $\Omega$,  up to translations, where $m^\star$ is the convex symmetrization of $m$.
%where $M=|\Omega|$ and $M_\lambda=|\Omega^\star_\lambda|$.
\end{prop}
\begin{proof}
    First of all we stress that, if $|\Omega|=|\Omega^\star_\lambda|$, then there is nothing to prove, since Faber-Krahn inequality implies that $u^*(s)=z^*(s)$.\\
    Moreover we have $u^*(|\Omega^\star_\lambda|)>z^*(|\Omega^\star_\lambda|)=0$. Then, the following definition is well posed
    \begin{equation*}
        s_0=\inf\{s\in [0,|\Omega^\star_\lambda|]\,:\; u^*(t)\ge z^*(t), \, \forall t \in [s,|\Omega^\star_\lambda|]\}.
    \end{equation*}
    By definition $u^*(s_0)=z^*(s_0)$ and we want to prove that $s_0=0$. We proceed by contradiction supposing that $s_0>0$. Then under this assumption,  $u^*$ and $z^*$ coincide in $0$ and $s_0$ and we have    \begin{equation}\label{ulezugez}
        \begin{cases}
            u^*(s)< z^*(s) & s\in (0,s_0)\\
            u^*(s)\ge z^*(s) & s\in (s_0,|\Omega^\star_\lambda|).
        \end{cases}
    \end{equation}
    By  \eqref{diffinequality}, \eqref{talenti-rad} and \eqref{ulezugez}, we have that
    \begin{equation*}
        -{u^*}'(t)\le -{z^*}'(t), \quad  \text{ for  every } t\in (0,s_0).
    \end{equation*}
    Integrating between $(0,s)$,  with  $s\in (0,s_0)$, being $u^*(0)=z^*(0)$, we get
    \begin{equation*}
        u^*(s)\ge z^*(s), \qquad \forall s \in (0,s_0),
    \end{equation*}
    which is in  contradiction with the definition of $s_0$. Hence $s_0=0$ and the proof is completed.
\end{proof}
As an immediately consequence of the previous result we get the following scale-invariant inequality for any $r>0$

\begin{equation} \label{infinitycomparison}
    \frac{\|u\|_{L^r(\Omega,m)}}{\|u\|_{L^\infty(\Omega)}}\ge \frac{\|z\|_{L^r(\Omega^{\star}_{\lambda},m^\star)}}{\|z\|_{L^\infty(\Omega_{\lambda}^{\star})}}.
\end{equation}

When the functions $u$ and $z$ are  normalized with respect the weighted  $L^q$-norm,  we get the following  comparison result.
\begin{thm}
Let $\Omega\subset\mathbb{R}^n$ be an open, bounded and connected set, $1<q\le p$ and let $m\in L^{\infty}(\Omega)$ be a positive function.  Let $u$ be a positive solution to the problem \eqref{APQ} corresponding to  $\lambda^H_{p,q}(\Omega)$ and let  $z$ be a positive eigenfunction to the problem \eqref{pbrad} corresponding to $\lambda^H_{p,q}(\Omega)$ such that  
 \begin{equation}\label{normq}
        \int_{\Omega} m\,u^{q}\,dx = \int_{\Omega^{\star}_\lambda} m^\star z^{q}\,dx.
    \end{equation}
Then  we have
    \begin{equation}
    \label{norma-r}
        \int_0^{s} m^* \,(u^*)^{r}\,dt\le \int_0^{s} m^*\,(z^*)^{r}\,dt, \quad s\in [0,|\Omega^\star_\lambda|], \quad q\le r 
    \end{equation}
    where $u^*$, $m^*$ and $z^*$ are respectively the decreasing rearrangements of $u$, $m$ and $z$, and $m^\star$ is the convex symmetrization of $m$. 
    The equality case holds if and only if $\Omega=\Omega^{\star}$, $z=u=u^*$ and $m=m^{\star}$ a.e. $\Omega$,  up to translations.
\end{thm}
\begin{proof}
If $|\Omega|=|\Omega^{\star}_\lambda|$ the conclusion is trivial. Let be $|\Omega|>|\Omega^{\star}_\lambda|$,  since $u$ and $z$ verify \eqref{normq}, by \eqref{infinitycomparison} it holds that 
    \begin{equation*}
        u^*(0)=\|u\|_{L^{\infty}(\Omega)} \le \|z\|_{L^{\infty}(\Omega^{\star}_{\lambda})}= z^*(0),
    \end{equation*}
     If $u^*(0)=z^*(0)$, then Proposition \ref{proppointwisecomp} and the normalization \eqref{normq} imply that $u^*(s)=z^*(s)$ for every $s\in [0,|\Omega^{\star}_\lambda|]$ and than the claim follows trivially.\\
    Let $u^*(0)<z^*(0)$. Since $u^*(|\Omega^{\star}_\lambda|)>z^*(|\Omega^{\star}_\lambda|)$, we can consider
    \begin{equation*}
        s_0 = \sup \{s\in (0,|\Omega^{\star}_\lambda|): u^*(t)\le z^*(t) \text{ for } t \in [0,s]\}.
    \end{equation*}
    Obviously  $0<s_0<|\Omega^{\star}_\lambda|$, $u^*(s_0)=z^*(s_0)$ and $u^*\le z^*$ in $[0,s_0]$.  
    We want to show that $u^*>z^*$ in $[s_0,|\Omega^{\star}_\lambda|]$. Indeed, if we suppose by contradiction that there exists  $s_1>s_0$ such that $u^*(s_1)=z^*(s_1)$ and $u^*(s)>z^*(s)$ for $s\in(s_0,s_1)$ we can construct the following function
    \begin{equation*}
        w^*(s)=
        \begin{cases}
            z^*(s) & s \in [0,s_0]\cup [s_1,|\Omega^{\star}_\lambda|]\\
            u^*(s) & s \in [s_0,s_1].
        \end{cases}
    \end{equation*}
%and define
 %   \begin{equation*}
  %      w(x)=w(k_nH^{\circ}(x)^n), \qquad x \in \Omega^\star_\lambda.
   % \end{equation*}
    It is straightforward to check that %$w\in W^{1,p}_0(\Omega^\star_\lambda)$. By property \eqref{HHcirc} and by the definition of $w$ we have that
    \begin{equation*}
        \int_{\Omega}H(\nabla w)^p\,dx = n^pk^p_n\int_{\Omega} (-{w^*}'(k_n H^{\circ}(x)^n))^pH^{\circ}(x)^{p(n-1)}\,dx.
    \end{equation*}
    Applying Coarea Formula and considering the change of variables $s=k_n t^n$, we get
    \begin{equation*}
        \int_{\Omega} H(\nabla w)^p\,dx = n^p k_n^{\frac{p}{n}}\int_0^{|\Omega^{\star}_\lambda|} s^{p-\frac{p}{n}}(-{w^*}'(t))^p\,dt.
    \end{equation*}
  %  On the other hand we have
   % \begin{equation*}
    %    \int_{\Omega}g^q\,dx = \int_0^{M_\lambda} w^q\,dt.
%    \end{equation*}
Thanks to the normalization \eqref{normq} and the definition of $w$, we have that
\begin{equation*}
    \| u\|_{L^q(\Omega,m)} = \| z\|_{L^q(\Omega^\star_\lambda,m^\star)} \le \| w\|_{L^q(\Omega^\star_\lambda,m^\star)},
\end{equation*}
then  by \eqref{diffinequality} and \eqref{talenti-rad} we have that
    \begin{equation} \label{diffinw}
        (-{w^*}'(s))^{p-1} \le n^{-p} k_n^{-\frac{p}{n}}\lambda^H_{p,q}(\Omega) \| w^*\|_{q,m^*}^{p-q}s^{\frac{p}{n}-p}\int_0^{s} m^*(r)({w^*})^{q-1}(r)\,dr.
    \end{equation}
    Multiplying \eqref{diffinw} by $-w'$, rearranging the terms and integrating between $0$ and $|\Omega^{\star}_\lambda|$, we get
    \begin{equation*}
    \begin{split}
        n^pk_n^{\frac{p}{n}}&\int_0^{|\Omega^{\star}_\lambda|} s^{p-\frac{p}{n}}(-{w^*}'(s))^p\,ds\le \\
        &\le \lambda^H_{p,q}(\Omega) \| w^*\|_{q,m^*}^{p-q} \int_0^{|\Omega^{\star}_\lambda|}(-{w^*}'(s))\int_0^sm^*(r)({w^*})^{q-1}(r)\,dr\,ds.    
    \end{split}
    \end{equation*}
    An integration by parts allows us to conclude that
    \begin{equation*}
        \frac{\ds\int_{\Omega^\star_\lambda} H(\nabla w)^p\,dx}{\bigg(\ds\int_{\Omega^\star_\lambda}m^\star w^q\,dx\bigg)^\frac{p}{q}}=\frac{\ds n^pk_n^{\frac{p}{n}}\int_0^{|\Omega^{\star}_\lambda|} s^{p-\frac{p}{n}}(-w'(s))^p\,ds}{\bigg(\ds \int_0^{|\Omega^{\star}_\lambda|}m^*(s)({w^*})^q(s)\,ds\bigg)^\frac{p}{q}}\le \lambda^H_{p,q}(\Omega)=\lambda^H_{p,q}(\Omega^{\star}_\lambda).
    \end{equation*}
   
    By the minimality and the simplicity of $\lambda^H_{p,q}$, and the definition of $w^*$, it must be $w^*(s)=z^*(s)$ for every $s\in[0,|\Omega^{\star}_\lambda|]$, but this is a contradiction since in $(s_0,s_1)$ we have that $u^*(s)>z^*(s)$.\\
    In this way we have proved that there exists a unique point $s_0$ where $u^*$ and $z^*$ can cross each other, and such that
    \begin{equation}\label{ulezleu}
        \begin{cases}
            u^*(s)\le z^*(s) & s\in [0,s_0]\\
            u^*(s)\ge z^*(s) & s\in [s_0,|\Omega^{\star}_\lambda|].
        \end{cases}
    \end{equation}
    If we extend $z^*$ to be zero in $[|\Omega^{\star}_\lambda|, |\Omega|]$, by \eqref{normq} and \eqref{ulezleu} then we have that for every $s\in [0,|\Omega|]$
    \begin{equation}\label{claim}
        \int_0^{s} m^*(t)({u^*(t)})^{q}\,dt\le \int_0^{s} m^*(t)({z^*(t)})^{q}\,dt. 
    \end{equation}
    Indeed \eqref{ulezleu} implies that the function 
    \[G(s)=\int_0^{s} {m^*}(t)((z^*))^{q}-(u^*)^q\,dt, \qquad s \in [0,|\Omega|]
    \]
    has a maximum in $s_0$ and cannot be negative in any point. This proves \eqref{claim}. Finally inequality \eqref{norma-r} follows easily by \eqref{claim} by using Proposition \ref{domi} being $m^{\star}(u^{\star})^q \prec m^{\star}z^q$. 
    \end{proof}
    %\begin{rem}
    %Then we can apply a generalization of a  result contained in \cite{hardy1929some} to say that for every $q\le r$
    %\begin{equation}\label{comparisonnormr}
     %   \int_0^{|\Omega|} m^*(s)u^*(s)^{r} \,ds \le \int_0^{|\Omega^{\star}_\lambda|} m^*(s)z^*(s)^{r}\,ds.
    %\end{equation}
    %that is the claim.
%\end{rem}
%\begin{cor}
%Let $u$ and $z$ as in Proposition \ref{proppointwisecomp}. For any $r>0$ there exists a positive constant $C=C(n,p,q,r,\lambda^{H}_{p,q})$ such that
%\begin{equation} \label{infinitycomparison}
 %    \|u\|_{L^\infty(\Omega)} \le C \|u\|_{L^r(\Omega,m)}.
%\end{equation}
%\end{cor}
%\begin{proof}
\begin{proof}[Proof of Theorem \ref{chitiAPQ}]
%\lambda_p(\Omega)^{-\frac{n}{p}(r-r)}
The proof of statement i) follows directly from \eqref{normq} and \eqref{norma-r}, indeed we have 
    \begin{equation*}
        \bigg(\int_{\Omega} m u^{r}\,dx\bigg)^{\frac{1}{r}}\le \bigg(\int_{\Omega^{\star}_\lambda} m^\star z^{r}\,dx\bigg)^{\frac{1}{r}} = \frac{\ds \bigg(\int_{\Omega^{\star}_\lambda} m^\star z^{r}\,dx\bigg)^{\frac{1}{r}} }{\ds \bigg(\int_{\Omega^{\star}_\lambda} m^\star z^{q}\,dx\bigg)^{\frac{1}{q}} } \bigg(\int_{\Omega} m u^{q}\,dx\bigg)^{\frac{1}{q}}.
    \end{equation*}
    Therefore we have
    \begin{equation*}
         \|u \|_{L^r(\Omega,m)} \le C\,\,\| u\|_{L^q(\Omega,m)}, \quad q\le r< +\infty;
    \end{equation*}
    with
    \begin{equation*}
        C = \ds\frac{\|z \|_{L^r(\Omega^{\star}_{\lambda},m^{\star})}}{\|z \|_{L^q(\Omega^{\star}_{\lambda},m^{\star})}}.
    \end{equation*}
The proof of statement ii) follows immediately by Proposition \ref{proppointwisecomp} with \\ $C= \frac{\|z\|_{\infty}}{\|z\|_{L^r(\Omega^\star,m^\star)}}$ and the theorem is completely proved.
\end{proof}

\section*{Acknowledgments}
This work has been partially supported by the MiUR PRIN-PNRR 2022 grant "Linear and Nonlinear PDE's: New directions and Applications" and by GNAMPA of INdAM. The second author acknowledges the MIUR Excellence Department Project awarded to the Department of Mathematics, University of Pisa, CUP I57G22000700001.

\section*{Conflicts of interest and data availability statement}
The authors declare that there is no conflict of interest. Data sharing not applicable to this article as no datasets were generated or analyzed during the current study.

\bibliographystyle{plain}
\bibliography{biblio}
\end{document}